\documentclass{article}
\usepackage{amsfonts}
\usepackage{mathrsfs}
\usepackage{amsthm}
\usepackage{amssymb}
\usepackage{amsmath}
\usepackage{enumerate}

\theoremstyle{plain}
\newtheorem{theorem}{Theorem}[section]
\newtheorem{proposition}[theorem]{Proposition}

\newtheorem{lemma}[theorem]{Lemma}

\theoremstyle{definition}
\newtheorem{definition}{Definition}[section]

\theoremstyle{remark}
\newtheorem{remark}{\textbf{Remark}}[section]

\theoremstyle{example}

\numberwithin{equation}{section}

\title{Spectral Integrals of Bernoulli Generalized Functionals}
\author{Jing Zhang, Caishi Wang\footnote{Author to whom correspondence should be addressed.
Electronic addresses: wangcs@nwnu.edu.cn}, Lu Zhang, Lixia Zhang\\
School of Mathematics and Statistics,
Northwest Normal University\\
Lanzhou 730070, People's Republic of China}

\begin{document}
\maketitle

\noindent\textbf{Abstract.}\ \
Let $\mathcal{S}\subset \mathcal{L}^2 \subset \mathcal{S}^*$ be the Gel'fand triple over the  Bernoulli space,
where elements of $\mathcal{S}^*$ are called Bernoulli generalized functionals.
In this paper, we define integrals of Bernoulli generalized functionals
with respect to a spectral measure (projection operator-valued measure)
in the framework of $\mathcal{S}\subset \mathcal{L}^2 \subset \mathcal{S}^*$,
and examine their fundamental properties.
New notions are introduced, several results are obtained and examples are also shown.
\vskip 2mm

\vskip 2mm


\section{Introduction}\label{sec-1}

A spectral measure is a projection operator-valued measure defined on a measurable space,
where a projection operator means an orthogonal projection operator on some Hilbert space.
Given a bounded measurable function $f$ and a spectral measure $\pi$ on a measurable space $(\mathfrak{X},\mathscr{F})$,
one can use the usual idea of Lebesgue integration to define an integral $\int_{\mathfrak{X}}fd\pi$,
which is a bounded linear operator on the Hilbert space associated with the spectral measure $\pi$ (see, e.g. \cite{parth}).
Such integrals are usually known as spectral integrals, which play an important role in the theory of operators \cite{helson, parth}.

Generalized functions (functionals) are continuous linear functionals on fundamental function spaces.
For instance, Schwartz generalized functions are continuous linear functionals on the Schwartz rapidly decreasing function space \cite{gel-1},
and Hida generalized functionals are continuous linear functionals on the Hida testing functional space \cite{hida,huang,kuo,obata}.
As is well known, generalized functions (functionals) have wide application in mathematical physics
(see, e.g. \cite{albe,barhoumi,benth,ji,huyz,lee,nunno} and references therein).

Given a generalized function (functional) $\Phi$ and a spectral measure $\pi$ on an appropriate measurable space $(\mathfrak{X},\mathscr{F})$,
one natural question arises: how can one define an integral $\int_{\mathfrak{X}}\Phi d\pi$ both reasonably and rigorously?
Such integrals are of physical significance \cite{accardi}.
However, the usual idea of Lebesgue integration does not work in this case,
which suggests that a new method is needed to define such an integral.
Accardi \textit{et al} \cite{accardi} considered such a question in the context of Hilbert space theory.
In 2005, by using Hida's white noise analysis \cite{hida}, Wang \textit{et al} \cite{wang-huang} introduced integrals
of Schwartz generalized functions with respect to a spectral measure on the Borel space $(\mathbb{R},\mathscr{B}(\mathbb{R}))$,
which were continuous linear operators from the space of Hida testing functionals to the space of Hida generalized functionals.
In particular, they gave a rigorous definition to the delta function $\delta(Q)$ of an observable $Q$ used in the physical literature.
It is known \cite{hida} that Hida's white noise analysis is essentially a theory of infinite dimensional stochastic analysis on
functionals of Brownian motion (also known as Gaussian white noise functionals).

Bernoulli functionals (also known as Rademacher functionals) are measurable functions defined on the Bernoulli space,
and can be viewed as functionals of Bernoulli process.
Much attention has been paid to Bernoulli functionals in the past fifteen years (see, e.g.\cite{krok,nourdin,privault,zheng} and references therein).
In 2014, by using the canonical orthonormal basis for the space of square integrable Bernoulli functionals,
Wang and Zhang \cite{wang-zhang} constructed a Gel'fand triple over the Bernoulli space,
which actually introduced Bernoulli generalized functionals.
The following year, Wang and Chen \cite{wang-chen-1} obtained a characterization theorem
for generalized functionals of discrete-time normal martingale via the Fock transform,
which covers the case of Bernoulli generalized functionals.
Recently, they have further considered operators acting on generalized functionals of discrete-time normal martingale \cite{wang-chen-2}.

Let $\mathcal{S}\subset \mathcal{L}^2 \subset \mathcal{S}^*$ be the Gel'fand triple over the Bernoulli space (see Section~\ref{sec-2} for details),
where elements of $\mathcal{S}^*$ are called Bernoulli generalized functionals.
In this paper, motivated by the work of \cite{wang-huang,wang-chen-2}, we would like to define integrals of Bernoulli generalized functionals
(namely elements of $S^*$) with respect to a spectral measure in the framework of $\mathcal{S}\subset \mathcal{L}^2 \subset \mathcal{S}^*$,
and examine their fundamental properties.

The paper is organized as follows. Section~\ref{sec-2} describes the Gel'fand triple $\mathcal{S}\subset \mathcal{L}^2 \subset \mathcal{S}^*$
over the Bernoulli space, which is the framework where we work.
In Section~\ref{sec-3}, we prove a technical theorem on the regularity of continuous linear operators from the space $\mathcal{S}$ of Bernoulli testing functionals to the space $\mathcal{S}^*$ of Bernoulli generalized functionals.
In Section~\ref{sec-4}, we first introduce a notion of $\mathcal{S}$-smooth spectral measures on the Bernoulli space, where $\mathcal{S}$ refers to
the space of Bernoulli testing functionals, and then, with the 2D-Fock transform as the main tool, we define
integrals of Bernoulli generalized functionals (namely elements of $\mathcal{S}^*$) with respect to an $\mathcal{S}$-smooth spectral measure,
which are actually continuous linear operators from $\mathcal{S}$ to $\mathcal{S}^*$.
We examine fundamental properties of these integrals, and establish a convergence theorem for sequences of these integrals.
Finally in Section~\ref{sec-5}, we show an example of an $\mathcal{S}$-smooth spectral measure and Bernoulli generalized functionals that
are integrable with this spectral measure. Some further results are also obtained therein.
\vskip 2mm

\textbf{Notation and conventions.} Throughout, $\mathbb{N}$ always denotes the set of all non-negative integers.
Unless otherwise stated, letters $j$, $k$ and $n$ stand for nonnegative integers.
We denote by $\Gamma$ the finite power set of $\mathbb{N}$, namely
$\Gamma=\{\sigma \mid \sigma\subset \mathbb{N}, \# \sigma<\infty\}$
where $\# \sigma$ means the cardinality of $\sigma$ as a set.

\section{Gel'fand triple over Bernoulli space}\label{sec-2}

This section describes the Gel'fand triple over the Bernoulli space, which is the framework where we work.

Recall that $\mathbb{N}$ denotes the set of all nonnegative integers.
Let $\Sigma=\{-1,1\}^{\mathbb{N}}$ be the set of all mappings $\omega\colon \mathbb{N} \rightarrow \{-1,1\}$, and
$(\zeta_n)_{n\geq 0}$ the sequence of canonical projections on $\Sigma$ given by
\begin{equation}\label{eq-2-1}
    \zeta_n(\omega)=\omega(n),\quad \omega\in \Sigma.
\end{equation}
Denote by $\mathscr{A}$ the $\sigma$-field on $\Sigma$ generated by the sequence $(\zeta_n)_{n\geq 0}$.
Let $(\theta_n)_{n\geq 0}$ be a given sequence of positive numbers with the property that $0 < \theta_n < 1$ for all $n\geq 0$.
It is known \cite{privault} that there exists a unique probability measure $\mu$ on $\mathscr{A}$ such that
\begin{equation}\label{eq-2-2}
\mu\circ(\zeta_{n_1}, \zeta_{n_2}, \cdots, \zeta_{n_k})^{-1}\big\{(\epsilon_1, \epsilon_2, \cdots, \epsilon_k)\big\}
=\prod_{j=1}^k \theta_{n_j}^{\frac{1+\epsilon_j}{2}}(1-\theta_{n_j})^{\frac{1-\epsilon_j}{2}}
\end{equation}
for $n_j\in \mathbb{N}$, $\epsilon_j\in \{-1,1\}$ ($1\leq j \leq k$) with $n_i\neq n_j$ when $i\neq j$
and $k\in \mathbb{N}$ with $k\geq 1$. Then we come to a probability measure space $(\Sigma, \mathscr{A}, \mu)$,
which is referred to as the \textbf{Bernoulli space}. Measurable functions (complex-valued random variables)
on $(\Sigma, \mathscr{A}, \mu)$ are usually known as \textbf{Bernoulli functionals}.

Let $(Z_n)_{n\geq 0}$ be the sequence of independent random variables on $(\Sigma, \mathscr{A}, \mu)$ defined by
\begin{equation}\label{eq-2-3}
   Z_n = \frac{\zeta_n + 1 - 2\theta_n}{2\sqrt{\theta_n(1-\theta_n)}},\quad n\geq0.
\end{equation}
Clearly, for each $n\geq 0$, $Z_n$ has a probability distribution of the following form
\begin{equation}\label{eq-2-4}
  \mu\left\{Z_n = \sqrt{(1-\theta_n)/\theta_n}\right\}=\theta_n,\quad
    \mu\left\{Z_n = -\sqrt{\theta_n/(1-\theta_n)}\right\}=1-\theta_n.
\end{equation}
Let $\mathcal{L}^2\equiv \mathcal{L}^2(\Sigma, \mathscr{A}, \mu)$ be the space of square integrable Bernoulli functionals.
We denote by $\langle\cdot,\cdot\rangle$ the usual inner product in space $\mathcal{L}^2$
given by
\begin{equation}\label{eq-2-5}
  \langle\xi,\eta\rangle=\int_{\Sigma}\overline{\xi}\eta d\mu,\quad \xi,\, \eta \in \mathcal{L}^2,
\end{equation}
and by $\|\cdot\|$ the corresponding norm.
It is known that $(Z_n)_{n\geq 0}$ has the chaotic representation property \cite{privault}.
Thus $\mathcal{L}^2$ has an orthonormal basis of the form $\{Z_{\sigma}\mid \sigma \in \Gamma\}$,
where $Z_{\emptyset}=1$ and
\begin{equation}\label{eq-2-6}
    Z_{\sigma} = \prod_{i\in \sigma}Z_i,\quad \text{$\sigma \in \Gamma$, $\sigma \neq \emptyset$},
\end{equation}
where $\Gamma$ is the finite power set of $\mathbb{N}$ (see Section~\ref{sec-1} for the definition of $\Gamma$).
Clearly, as a complex Hilbert space, $\mathcal{L}^2$ is infinite-dimensional and separable.
In what follows, we call $\{Z_{\sigma}\mid \sigma \in \Gamma\}$ the \textbf{canonical orthonormal basis} for $\mathcal{L}^2$.

\begin{lemma}\label{lem-2-1}\cite{wang-chen-1}
Let $\sigma \mapsto \lambda_{\sigma}$ be the positive integer-valued function on $\Gamma$ given by
\begin{equation}\label{eq-2-7}
  \lambda_{\sigma}
     =\left\{
     \begin{array}{ll}
       1, & \hbox{$\sigma=\emptyset$, $\sigma\in \Gamma$;} \\
       \Pi_{k\in \sigma}(1+k), & \hbox{$\sigma\neq\emptyset$, $\sigma\in \Gamma$.}
     \end{array}
   \right.
\end{equation}
Then the series $\sum_{\sigma\in \Gamma}\lambda_{\sigma}^{-r}$ converges for all real number $r>1$, and moreover,
it holds true that
\begin{equation}\label{eq-2-8}
  \sum_{\sigma\in \Gamma}\lambda_{\sigma}^{-r} \leq \exp\Big[\sum_{n=1}^{\infty}n^{-r}\Big].
\end{equation}
\end{lemma}

Using the function $\sigma\mapsto \lambda_{\sigma}$ introduced above, we can construct a chain of Hilbert spaces
consisting of Bernoulli functionals as follows.

For $\sigma\in \Gamma$, we use $|Z_{\sigma}\rangle\!\langle Z_{\sigma}|$ to mean the Dirac operator associated with
the basis vector $Z_{\sigma}$, which is a $1$-dimensional projection operator on $\mathcal{L}^2$.
Then the countable family $\{\,|Z_{\sigma}\rangle\!\langle Z_{\sigma}|\,\}_{\sigma\in \Gamma}$ forms
a resolution of identity on $\mathcal{L}^2$.
For $p \geq 0$, let $\mathcal{S}_p$ be the domain of the operator
$A_p =  \sum_{\sigma\in \Gamma}\lambda_{\sigma}^p|Z_{\sigma}\rangle\!\langle Z_{\sigma}|$, namely
\begin{equation}\label{eq-2-9}
  \mathcal{S}_p = \mathrm{Dom}\, A_p
=\Big\{\xi\in \mathcal{L}^2 \Bigm| \sum_{\sigma\in \Gamma}\lambda_{\sigma}^{2p}|\langle Z_{\sigma}, \xi\rangle|^2<\infty\Big\}.
\end{equation}
It is easy to verify that $\mathcal{S}_p$ becomes a Hilbert space with the inner product $\langle\cdot,\cdot\rangle_p$ given by
\begin{equation}\label{eq-2-10}
  \langle\xi,\eta\rangle_p = \langle A_p\xi,  A_p\eta\rangle
  = \sum_{\sigma\in \Gamma}\lambda_{\sigma}^{2p}\overline{\langle Z_{\sigma}, \xi\rangle}\langle Z_{\sigma},\eta\rangle,\quad \xi,\, \eta\in\mathcal{S}_p.
\end{equation}
We denote by $\|\cdot\|_p$ the norm induced by $\langle\cdot,\cdot\rangle_p$, which obviously satisfies the following relations
\begin{equation}\label{eq-2-11}
  \|\xi\|_p^2 = \|A_p\xi\|^2 = \sum_{\sigma\in \Gamma}\lambda_{\sigma}^{2p}|\langle Z_{\sigma}, \xi\rangle|^2,\quad \xi\in\mathcal{S}_p.
\end{equation}

\begin{lemma}\cite{wang-chen-1} \label{lem-2-2}
Let $p \geq  0$ be given. Then $\{Z_{\sigma} | \sigma \in \Gamma\} \subset \mathcal{S}_p$ and, moreover, the system
$\{\lambda_{\sigma}^{-p}Z_{\sigma}\mid \sigma\in \Gamma\} $ forms an orthonormal basis for $\mathcal{S}_p$.
\end{lemma}

It is not hard to show that the norms $\{\|\cdot\|_p \mid p\geq 0\}$ are compatible.
This, together with the fact $\lambda_{\sigma}\geq 1$ for all $\sigma\in \Gamma$,
implies that $\|\cdot\|_p\leq \|\cdot\|_q$
and $\mathcal{S}_q\subset \mathcal{S}_p$  whenever $0 \leq p \leq q$. Consequently, we get a chain of Hilbert spaces of Bernoulli functionals
as follows:
\begin{equation}\label{eq-2-12}
 \cdots \subset \mathcal{S}_{p+1} \subset \mathcal{S}_p \subset \cdots \subset \mathcal{S}_0 =\mathcal{L}^2.
\end{equation}
We put
\begin{equation}\label{eq-2-13}
  \mathcal{S} = \bigcap_{p=0}^{\infty}\mathcal{S}_p
\end{equation}
and endow it with the topology generated by the norm sequence $\big(\|\cdot\|_p\big)_{p \geq 0}$.
Note that, for each $p \geq 0$, $\mathcal{S}_p$ is just the completion of $\mathcal{S}$ with respect to norm $\|\cdot\|_p$.
Thus $\mathcal{S}$ is a countably-Hilbert space. The next lemma, however, shows that $\mathcal{S}$ even has a much
better property.

\begin{lemma}\cite{wang-chen-1} \label{lem-2-3}
The space $\mathcal{S}$ is a nuclear space, namely for any $p\geq 0$, there exists $q > p$ such that the inclusion mapping
$i_{pq} \colon \mathcal{S}_q \rightarrow \mathcal{S}_p$ defined by $i_{pq}(\xi) = \xi$ is a Hilbert-Schmidt operator.
\end{lemma}

For $p \geq 0$, we denote by $\mathcal{S}_p^*$ the dual of $\mathcal{S}_p$ and $\|\cdot\|_{-p}$ the norm of $\mathcal{S}_p^*$.
Then $\mathcal{S}_p^*\subset \mathcal{S}_q^*$ and  $\|\cdot\|_{-p}\geq \|\cdot\|_{-q}$ whenever $0 \leq p \leq q$.
The lemma below is then an immediate consequence of the general theory of countably-Hilbert spaces (see \cite{gel-1,becnel}).

\begin{lemma}\label{lem-2-4}
Let $\mathcal{S}^*$ be the dual of $\mathcal{S}$ and endow it with the strong topology. Then
\begin{equation}\label{eq-2-14}
  \mathcal{S}^* = \bigcup_{p=0}^{\infty}\mathcal{S}^*_p
\end{equation}
and, moreover, the inductive limit topology on $\mathcal{S}^*$ given by space sequence $\{\mathcal{S}^*_p\}_{p\geq 0}$ coincides
with the strong topology.
\end{lemma}

By identifying $\mathcal{L}^2$ with its dual, one naturally comes to a Gel'fand triple of the following form
\begin{equation}\label{eq-2-15}
\mathcal{S}\subset \mathcal{L}^2 \subset \mathcal{S}^*,
\end{equation}
which is referred to as the \textbf{Gel'fand triple over the Bernoulli space} $(\Sigma,\mathscr{A},\mu)$.
By convention, elements of $\mathcal{S}^*$ are called \textbf{Bernoulli generalized functionals},
while elements of $\mathcal{S}$ are called \textbf{Bernoulli testing functionals}.

\begin{lemma}\label{lem-2-5}\cite{wang-chen-1}
The system $\{Z_{\sigma} \mid \sigma \in \Gamma\}$  is contained in $\mathcal{S}$ and, moreover, it forms a
basis for $\mathcal{S}$ in the sense that
\begin{equation}\label{eq-2-16}
  \xi = \sum_{\sigma\in \Gamma}\langle Z_{\sigma}, \xi\rangle Z_{\sigma},\quad \xi\in \mathcal{S},
\end{equation}
where $\langle\cdot,\cdot\rangle$ is the inner product of $\mathcal{L}^2$ and the series converges in the topology of $\mathcal{S}$.
\end{lemma}

We denote by $\langle\!\langle\cdot,\cdot\rangle\!\rangle$ the canonical bilinear form (also known as pairing) on $\mathcal{S}^*\times \mathcal{S}$, namely
\begin{equation}\label{eq-2-17}
  \langle\!\langle\Phi,\xi\rangle\!\rangle = \Phi(\xi),\quad \Phi\in \mathcal{S}^*,\, \xi\in \mathcal{S},
\end{equation}
where $\Phi(\xi)$ means the value of the functional $\Phi$ at $\xi$. Note that $\langle\cdot,\cdot\rangle$ denotes the inner product of $\mathcal{L}^2$,
which is different from $\langle\!\langle\cdot,\cdot\rangle\!\rangle$.

\begin{lemma}\cite{wang-lin}\label{lem-2-6}
Let $\Phi\in \mathcal{S}^*$ be given. Then, for $p\geq 0$, $\Phi\in \mathcal{S}^*_p$ if and only if $\Phi$ satisfies that
\begin{equation}\label{eq-2-18}
  \sum_{\sigma\in \Gamma}\lambda_{\sigma}^{-2p}|\langle\!\langle \Phi, Z_{\sigma}\rangle\!\rangle|^2<\infty.
\end{equation}
In that case $\|\Phi\|_{-p}^2 = \sum_{\sigma\in \Gamma}\lambda_{\sigma}^{-2p}|\langle\!\langle \Phi, Z_{\sigma}\rangle\!\rangle|^2$.
\end{lemma}

\section{Technical theorem}\label{sec-3}

In this section, we establish a technical theorem about the regularity of
operators acting on Bernoulli functionals, which will be used to prove our main results.
We keep using the notions and notation fixed in previous sections.

\begin{definition}\cite{wang-chen-2}\label{def-3-1}
For an operator $\mathsf{T}\colon \mathcal{S} \rightarrow \mathcal{S}^*$, its 2D-Fock transform is the function
$\widehat{\mathsf{T}}$ on $\Gamma\times \Gamma$ given by
\begin{equation}\label{eq-3-1}
  \widehat{\mathsf{T}}(\sigma,\tau)=\langle\!\langle \mathsf{T}Z_{\sigma},Z_{\tau}\rangle\!\rangle,\quad \sigma,\, \tau \in \Gamma.
\end{equation}
\end{definition}

Continuous linear operators from $\mathcal{S}$ to $\mathcal{S}^*$ are completely determined by their 2D-Fock transforms.
More precisely, if $\mathsf{T}_1$, $\mathsf{T}_2\colon \mathcal{S}\rightarrow \mathcal{S}^*$ are continuous linear operators,
then $\mathsf{T}_1=\mathsf{T}_2$
if and only if their 2D-Fock transforms are the same, namely $\widehat{\mathsf{T}_1}=\widehat{\mathsf{T}_2}$.
The following lemma offers a useful characterization of
continuous linear operators from $\mathcal{S}$ to $\mathcal{S}^*$ via their 2D-Fock transforms.

\begin{lemma}\cite{wang-chen-2}\label{lem-3-1}
A function $G$ on $\Gamma\times \Gamma$ is the 2D-Fock transform of a continuous linear operator $\mathsf{T}\colon \mathcal{S} \rightarrow \mathcal{S}^*$
if and only if it satisfies that
\begin{equation}\label{eq-3-2}
  |G(\sigma, \tau)|\leq C\lambda_{\sigma}^p\lambda_{\tau}^p,\quad  \sigma,\, \tau \in \Gamma
\end{equation}
for some constants $C\geq 0$ and $p\geq 0$.
\end{lemma}

For $p\geq 0$, we denote by $\mathfrak{L}(\mathcal{S}_p,\mathcal{S}_p^*)$ the Banach space of all bounded linear operators
from $\mathcal{S}_p$ to $\mathcal{S}_p^*$ and by $\|\cdot\|_{\mathfrak{L}(\mathcal{S}_p,\mathcal{S}_p^*)}$ the usual operator norm
in $\mathfrak{L}(\mathcal{S}_p,\mathcal{S}_p^*)$, which is given by
\begin{equation}\label{eq-3-3}
  \|\mathsf{T}\|_{\mathfrak{L}(\mathcal{S}_p,\mathcal{S}_p^*)}
  = \sup\{ \|\mathsf{T}\xi\|_{-p} \mid \xi\in \mathcal{S}_p,\, \|\xi\|_p = 1 \},\quad \mathsf{T}\in \mathfrak{L}(\mathcal{S}_p,\mathcal{S}_p^*).
\end{equation}
Note that $\mathcal{S}$ is dense in $\mathcal{S}_p$.
Thus, for each bounded linear operator $\mathsf{A} \colon (\mathcal{S}, \|\cdot\|_p)\rightarrow \mathcal{S}_p^*$, there exists
a unique bounded linear operator $\widetilde{\mathsf{A}}\in \mathfrak{L}(\mathcal{S}_p,\mathcal{S}_p^*)$ such that
$\widetilde{\mathsf{A}}\xi = \mathsf{A}\xi$,\ $\forall\, \xi\in \mathcal{S}$ and
\begin{equation*}
  \|\widetilde{\mathsf{A}}\|_{\mathfrak{L}(\mathcal{S}_p,\mathcal{S}_p^*)}
  = \sup\{\|\mathsf{A}\xi\|_{-p} \mid \xi\in \mathcal{S},\, \|\xi\|_p = 1\}.
\end{equation*}
The operator $\widetilde{\mathsf{A}}$ is usually known as the norm-keeping extension of the operator $\mathsf{A}$ to $\mathcal{S}_p$.

The next theorem is a result about the regularity of continuous linear operator from $\mathcal{S}$ to $\mathcal{S}^*$,
which will play an important role in proving our main results.

\begin{theorem}\label{thr-3-2}
Let $\mathsf{T}\colon \mathcal{S}\rightarrow \mathcal{S}^*$ be a continuous linear operator.
Suppose that $\mathsf{T}$ satisfies
\begin{equation}\label{eq-3-4}
 |\widehat{\mathsf{T}}(\sigma,\tau)|
  \leq C\lambda_{\sigma}^p\lambda_{\tau}^p,\quad  \sigma,\, \tau \in \Gamma
\end{equation}
for some constants $C\geq 0$ and $p\geq 0$. Then, for $q> p+\frac{1}{2}$, there exists a unique $\widetilde{\mathsf{T}} \in \mathfrak{L}(\mathcal{S}_q,\mathcal{S}_q^*)$
such that
\begin{equation}\label{eq-3-5}
  \big\|\widetilde{\mathsf{T}}\big\|_{\mathfrak{L}(\mathcal{S}_q,\mathcal{S}_q^*)}\leq C\sum_{\sigma\in \Gamma}\lambda_{\sigma}^{-2(q-p)}
\end{equation}
and $\widetilde{\mathsf{T}}\xi =\mathsf{T}\xi$\ for all\ $\xi\in \mathcal{S}$.
\end{theorem}

\begin{proof}
By Lemma~\ref{lem-2-1}, we know that $\sum_{\tau\in \Gamma}\lambda_{\tau}^{-2(q-p)}<\infty$ since $2(q-p)>1$.
Let $\sigma\in \Gamma$ be given. Then, $\mathsf{T}Z_{\sigma}\in \mathcal{S}^*$ and, by using the assumption (\ref{eq-3-4}),
we find
\begin{equation*}
  \sum_{\tau\in \Gamma}\lambda_{\tau}^{-2q}|\langle\!\langle \mathsf{T} Z_{\sigma}, Z_{\tau} \rangle\!\rangle|^2
  = \sum_{\tau\in \Gamma}\lambda_{\tau}^{-2q}|\widehat{\mathsf{T}} (\sigma, \tau)|^2
  \leq C^2\lambda_{\sigma}^{2p}\sum_{\tau\in \Gamma}\lambda_{\tau}^{-2(q-p)}
  <\infty,
\end{equation*}
which, together with Lemma~\ref{lem-2-6}, implies that $\mathsf{T}Z_{\sigma}\in \mathcal{S}_q^*$ and
\begin{equation*}
  \|\mathsf{T}Z_{\sigma}\|_{-q}^2
  = \sum_{\tau\in \Gamma}\lambda_{\tau}^{-2q}|\langle\!\langle \mathsf{T} Z_{\sigma}, Z_{\tau} \rangle\!\rangle|^2
  \leq C^2\lambda_{\sigma}^{2p}\sum_{\tau\in \Gamma}\lambda_{\tau}^{-2(q-p)}.
\end{equation*}
Now take $\xi\in \mathcal{S}$. Then $\sum_{\sigma\in \Gamma}\langle Z_{\sigma}, \xi\rangle \mathsf{T}Z_{\sigma}$ is a series in $\mathcal{S}_q^*$.
And, by using the above inequality, we have
\begin{equation*}
\begin{split}
  \sum_{\sigma\in \Gamma}\|\langle Z_{\sigma}, \xi\rangle \mathsf{T}Z_{\sigma}\|_{-q}
   & \leq \Big[\sum_{\sigma\in \Gamma} \lambda_{\sigma}^{2q}|\langle Z_{\sigma}, \xi\rangle|^2\Big]^{\frac{1}{2}}
      \Big[\sum_{\sigma\in \Gamma} \lambda_{\sigma}^{-2q}\|\mathsf{T}Z_{\sigma}\|_{-q}^2\Big]^{\frac{1}{2}}\\
   & \leq \|\xi\|_q \Big[\sum_{\sigma\in \Gamma} \lambda_{\sigma}^{-2q}C^2\lambda_{\sigma}^{2p}\sum_{\tau\in \Gamma}\lambda_{\tau}^{-2(q-p)}\Big]^{\frac{1}{2}}\\
   & = C \Big[\sum_{\sigma\in \Gamma}\lambda_{\sigma}^{-2(q-p)}\Big] \|\xi\|_q,
\end{split}
\end{equation*}
which implies that the series $\sum_{\sigma\in \Gamma}\langle Z_{\sigma}, \xi\rangle \mathsf{T}Z_{\sigma}$ converges in $\mathcal{S}_q^*$,
hence its sum $\sum_{\sigma\in \Gamma}\langle Z_{\sigma}, \xi\rangle \mathsf{T}Z_{\sigma}$ belongs to $\mathcal{S}_q^*$.
On the other hand, by Lemma~\ref{lem-2-5} and the continuity of $\mathsf{T}\colon \mathcal{S} \rightarrow \mathcal{S}^*$, we can get
\begin{equation*}
  \mathsf{T}\xi = \sum_{\sigma\in \Gamma}\langle Z_{\sigma}, \xi\rangle \mathsf{T}Z_{\sigma}.
\end{equation*}
Thus $\mathsf{T}\xi\in \mathcal{S}_q^*$ and
\begin{equation*}
  \|\mathsf{T}\xi\|_{-q}
  \leq \sum_{\sigma\in \Gamma}\|\langle Z_{\sigma}, \xi\rangle \mathsf{T}Z_{\sigma}\|_{-q}
  \leq C \Big[\sum_{\sigma\in \Gamma}\lambda_{\sigma}^{-2(q-p)}\Big] \|\xi\|_q,
\end{equation*}
which, together with the arbitrariness of $\xi\in \mathcal{S}$, implies that $\mathsf{T}$ is a bounded linear operator from $(\mathcal{S},\|\cdot\|_q)$
to $\mathcal{S}_q^*$. Therefore, there exists a unique $\widetilde{\mathsf{T}} \in \mathfrak{L}(\mathcal{S}_q,\mathcal{S}_q^*)$
such that $\widetilde{\mathsf{T}}\xi =\mathsf{T}\xi$, $x\in \mathcal{S}$ and
\begin{equation*}
  \big\|\widetilde{\mathsf{T}}\big\|_{\mathfrak{L}(\mathcal{S}_q,\mathcal{S}_q^*)}
  = \sup\{\|\mathsf{T}\xi\|_{-q} \mid \xi \in S,\, \|\xi\|_q=1\}
  \leq C \sum_{\sigma\in \Gamma}\lambda_{\sigma}^{-2(q-p)}.
\end{equation*}
This completes the proof.
\end{proof}

\section{Spectral integrals of Bernoulli generalized functionals}\label{sec-4}

In the present section, we define integrals of Bernoulli generalized functionals with respect to a spectral measure on the Bernoulli space
and examine their fundamental properties.

We continue to use the notation fixed in previous sections.
Additionally, we denote by $\mathfrak{P}(\mathcal{L}^2)$ the set of all projection operators on $\mathcal{L}^2$, which is a subset of the Banach algebra $\mathfrak{B}(\mathcal{L}^2)$
of all bounded linear operators on $\mathcal{L}^2$,
and by $\mathfrak{L}(\mathcal{S},\mathcal{S}^*)$ the space of all continuous linear operators
from $\mathcal{S}$ to $\mathcal{S}^*$.

Recall that the Bernoulli space $(\Sigma, \mathscr{A}, \mu)$ is actually a probability measure space.
This naturally leads to the next definition.

\begin{definition}\label{def-4-1}
A mapping $\pi\colon \mathscr{A}\rightarrow \mathfrak{P}(\mathcal{L}^2)$ is called a spectral measure
on $(\Sigma, \mathscr{A}, \mu)$ if it satisfies the following two requirements:
\begin{enumerate}
  \item[(1)] $\pi(\Sigma)=I$, where $I$ denotes the identity operator on $\mathcal{L}^2$;
  \item[(2)] For each sequence $(E_n)_{n\geq 1}\subset \mathscr{A}$ with $E_m \cap E_n = \emptyset$ when $m\neq n$, it holds true that
  \begin{equation}\label{eq-4-1}
    \pi\Big(\bigcup_{n=1}^{\infty}E_n\Big) = \sum_{n=1}^{\infty}\pi(E_n),
  \end{equation}
\end{enumerate}
where the operator series on the right-hand side converges strongly, namely in the strong operator topology of $\mathfrak{B}(\mathcal{L}^2)$.
\end{definition}

A spectral measure admits many interesting properties. The next lemma just shows the most striking one,
which is well known in the theory of functional analysis (see, e.g. \cite{parth}).

\begin{lemma}\label{lem-4-1}
If $\pi\colon \mathscr{A} \rightarrow \mathfrak{P}(\mathcal{L}^2)$ is a spectral measure on $(\Sigma,\mathscr{A},\mu)$, then for all
$E_1$, $E_2\in \mathscr{A}$ it holds true that
\begin{equation}\label{eq-4-2}
\pi(E_1\cap E_2)=\pi(E_1)\pi(E_2),
\end{equation}
where $\pi(E_1)\pi(E_2)$ just means the usual composition of operators $\pi(E_1)$ and $\pi(E_2)$.
\end{lemma}

Let $\pi\colon \mathscr{A} \rightarrow \mathfrak{P}(\mathcal{L}^2)$ be a spectral measure on $(\Sigma,\mathscr{A},\mu)$. Then,
for fixed $\xi$, $\eta\in \mathcal{L}^2$, the function
\begin{equation*}
E\mapsto \langle \pi(E)\xi, \eta\rangle
\end{equation*}
defines a complex-valued measure on the measurable space $(\Sigma,\mathscr{A})$.
In particular, for $\sigma$, $\tau\in \Gamma$, the function
$E\mapsto \langle \pi(E)Z_{\sigma}, Z_{\tau}\rangle$ is a complex-valued measure on $(\Sigma,\mathscr{A})$.

\begin{definition}\label{def-4-2}
A spectral measure $\pi\colon \mathscr{A} \rightarrow \mathfrak{P}(\mathcal{L}^2)$ on $(\Sigma,\mathscr{A},\mu)$ is said
to be $\mathcal{S}$-smooth if for each pair $(\sigma,\tau)\in \Gamma\times \Gamma$ there exists a Bernoulli testing functional
$\phi_{\sigma,\tau}^{\pi}\in \mathcal{S}$ such that
\begin{equation}\label{eq-4-3}
  \langle \pi(E) Z_{\sigma}, Z_{\tau}\rangle = \int_E \phi_{\sigma,\tau}^{\pi}d\mu,\quad \forall\, E\in \mathscr{A}.
\end{equation}
In that case, $\phi_{\sigma,\tau}^{\pi}$ is called the numerical density of $\pi$ associated with $(\sigma, \tau)\in \Gamma\times \Gamma$.
\end{definition}

For a nonnegative integer $n\geq 0$, we write $\Gamma\!_n=\{\sigma \mid \sigma\subset \mathbb{N}_n\}$, where $\mathbb{N}_n=\{0,1,\cdots, n\}$.
Obviously, $\Gamma\!_n\subset \Gamma\!_{n+1}\subset \Gamma$ for all $n\geq 0$, and $\bigcup_{n=0}^{\infty}\Gamma\!_n=\Gamma$.
In particular, $\Gamma\!_n$ has exactly $2^{n+1}$ elements.

\begin{proposition}\label{prop-4-2}
Let $\pi\colon \mathscr{A} \rightarrow \mathfrak{P}(\mathcal{L}^2)$ be a $\mathcal{S}$-smooth spectral measure on $(\Sigma,\mathscr{A},\mu)$
and $\phi_{\sigma,\tau}^{\pi}$ its numerical spectral density
associated with $(\sigma,\tau)\in \Gamma\times \Gamma$. Then, for all $n\geq 0$ and all $\xi\in \mathcal{L}^2$, it holds true that
\begin{equation}\label{eq-4-4}
  \sum_{\sigma, \tau \in \Gamma\!_n}\overline{\langle Z_{\sigma}, \xi\rangle}\langle Z_{\tau}, \xi\rangle \phi_{\sigma,\tau}^{\pi}\geq 0\quad
  \mbox{$\mu$-a.e. in $\Sigma$},
\end{equation}
where $\sum_{\sigma, \tau \in \Gamma\!_n}\overline{\langle Z_{\sigma}, \xi\rangle}\langle Z_{\tau}, \xi\rangle \phi_{\sigma,\tau}^{\pi}$ is viewed as a function on $\Sigma$.
\end{proposition}

\begin{proof}
Let $n\geq 0$ and $\xi\in \mathcal{L}^2$ be given. Then, for any $E\in \mathscr{A}$, by using the fact of $\pi(E)$ being a projection operator on $\mathcal{L}^2$ we have
\begin{equation*}
\begin{split}
  \int_E\Big(\sum_{\sigma, \tau \in \Gamma\!_n}\overline{\langle Z_{\sigma}, \xi\rangle}\langle Z_{\tau}, \xi\rangle \phi_{\sigma,\tau}^{\pi}\Big)d\mu
  & = \sum_{\sigma, \tau \in \Gamma\!_n}\overline{\langle Z_{\sigma}, \xi\rangle}\langle Z_{\tau}, \xi\rangle \int_E \phi_{\sigma,\tau}^{\pi}d\mu\\
  & = \sum_{\sigma, \tau \in \Gamma\!_n}\overline{\langle Z_{\sigma}, \xi\rangle}\langle Z_{\tau}, \xi\rangle \langle \pi(E) Z_{\sigma}, Z_{\tau}\rangle\\
  & = \Big\langle \pi(E) \sum_{\sigma\in \Gamma\!_n}\langle Z_{\sigma}, \xi\rangle Z_{\sigma}, \sum_{\tau \in \Gamma\!_{n}}\langle Z_{\tau}, \xi\rangle Z_{\tau}\Big\rangle\\
  & = \Big\|\pi(E) \sum_{\sigma\in \Gamma\!_n}\langle Z_{\sigma}, \xi\rangle Z_{\sigma}\Big\|^2\\
  & \geq 0,
\end{split}
\end{equation*}
which together with the arbitrariness of $E\in \mathscr{A}$ implies that
\begin{equation*}
  \sum_{\sigma, \tau \in \Gamma\!_n}\overline{\langle Z_{\sigma}, \xi\rangle}\langle Z_{\tau}, \xi\rangle \phi_{\sigma,\tau}^{\pi}\geq 0\quad
  \mbox{$\mu$-a.e. in $\Sigma$},
\end{equation*}
namely, as a function on $\Sigma$,
$\sum_{\sigma, \tau \in \Gamma\!_n}\overline{\langle Z_{\sigma}, \xi\rangle}\langle Z_{\tau}, \xi\rangle \phi_{\sigma,\tau}^{\pi}$
takes nonnegative values at almost all points in $\Sigma$.
\end{proof}

\begin{definition}\label{def-4-3}
Let $\pi\colon \mathscr{A} \rightarrow \mathfrak{P}(\mathcal{L}^2)$ be a $\mathcal{S}$-smooth spectral measure on $(\Sigma,\mathscr{A},\mu)$
and $\phi_{\sigma,\tau}^{\pi}$ its numerical spectral density
associated with $(\sigma,\tau)\in \Gamma\times \Gamma$. A Bernoulli generalized functional $\Phi \in \mathcal{S}^*$ is said to be integrable with respect to $\pi$ if there exist constants $C\geq 0$ and $p\geq 0$ such that
\begin{equation}\label{eq-4-5}
  \big|\big\langle\!\big\langle \Phi, \phi_{\sigma,\tau}^{\pi}\big\rangle\!\big\rangle\big|\leq C\lambda_{\sigma}^p\lambda_{\tau}^p,\quad
  \forall\, (\sigma,\tau)\in \Gamma\times \Gamma.
\end{equation}
In that case, by Lemma~\ref{lem-3-1}, there exists a unique operator $\mathsf{T}_{\Phi,\pi}\in \mathfrak{L}(\mathcal{S},\mathcal{S}^*)$ such that
\begin{equation}\label{eq-4-6}
  \widehat{\mathsf{T}_{\Phi,\pi}}(\sigma, \tau)
  = \big\langle\!\big\langle \Phi, \phi_{\sigma,\tau}^{\pi}\big\rangle\!\big\rangle,\quad
  \forall\, (\sigma,\tau)\in \Gamma\times \Gamma.
\end{equation}
We call $\mathsf{T}_{\Phi,\pi}$ the spectral integral of $\Phi$ with respect to $\pi$ and write $\int_{\Sigma}\Phi d\pi=\mathsf{T}_{\Phi,\pi}$.
\end{definition}

In the rest of the present section, we always assume that $\pi\colon \mathscr{A} \rightarrow \mathfrak{P}(\mathcal{L}^2)$ is
a fixed $\mathcal{S}$-smooth spectral measure on $(\Sigma,\mathscr{A},\mu)$
and $\phi_{\sigma,\tau}^{\pi}$ its numerical spectral density associated with $(\sigma,\tau)\in \Gamma\times \Gamma$.
Thus, if a Bernoulli generalized functional $\Phi$ is integrable with respect to $\pi$,
then its spectral integral $\int_{\Sigma} \Phi d\pi$ is a continuous linear operator from $\mathcal{S}$ to $\mathcal{S}^*$,
namely $\int_{\Sigma} \Phi d\pi\in \mathfrak{L}(\mathcal{S},\mathcal{S}^*)$, and satisfies that
\begin{equation*}
  \widehat{\int_{\Sigma} \Phi d\pi}(\sigma,\tau)
  = \big\langle\!\big\langle \big(\int_{\Sigma} \Phi d\pi\big)Z_{\sigma}, Z_{\tau}\big\rangle\!\big\rangle
  = \big\langle\!\big\langle \Phi, \phi_{\sigma,\tau}^{\pi}\big\rangle\!\big\rangle,\quad (\sigma,\tau)\in \Gamma\times \Gamma.
\end{equation*}

\begin{remark}\label{rem-4-1}
Let $\varphi\in \mathcal{L}^2$ be a bounded function on $\Sigma$ and $\int_{\Sigma}\varphi d\pi$ be the usual spectral integral
of $\varphi$ with respect to $\pi$, which is a bounded linear operator
on $\mathcal{L}^2$. Suppose that $\mathsf{R}_0\varphi$ is integrable with respect to $\pi$
in the sense of Definition~\ref{def-4-3}, where
$\mathsf{R}_0\colon \mathcal{L}^2\rightarrow (\mathcal{L}^2)^*$ denotes the Riesz mapping.
Then, the spectral integral $\int_{\Sigma} \mathsf{R}_0\varphi d\pi$ in the sense of Definition~\ref{def-4-3}
admits the following features
\begin{equation*}
  \int_{\Sigma} \mathsf{R}_0\varphi d\pi = \mathsf{R}_0\int_{\Sigma} \varphi d\pi,
\end{equation*}
where $\mathsf{R}_0\int_{\Sigma} \varphi d\pi$ means the composition of operators $\mathsf{R}_0$
and $\int_{\Sigma} \varphi d\pi$. This justifies our Definition~\ref{def-4-3}.
\end{remark}

\begin{theorem}\label{thr-4-3}
Let $\Phi$, $\Psi \in \mathcal{S}^*$ be integrable with respect to $\pi$. Then, for all $\alpha$, $\beta\in \mathbb{C}$,
$\alpha \Phi + \beta \Psi$ remains integrable with respect to $\pi$, and moreover it holds true that
\begin{equation}\label{eq-4-7}
  \int_{\Sigma}\big(\alpha \Phi + \beta \Psi\big)d\pi =  \alpha\int_{\Sigma} \Phi d\pi  +  \beta\int_{\Sigma}\Psi d\pi.
\end{equation}
\end{theorem}

\begin{proof}
It follows from the integrability of $\Phi$ and $\Psi$ that there exist nonnegative constants $C_1$, $C_2$, $p_1$ and $p_2$ such that
\begin{equation*}
  \big|\big\langle\!\big\langle \Phi, \phi_{\sigma,\tau}^{\pi}\big\rangle\!\big\rangle\big|\leq C_1\lambda_{\sigma}^{p_1}\lambda_{\tau}^{p_1}\quad
  \mbox{and}\quad
  \big|\big\langle\!\big\langle \Psi, \phi_{\sigma,\tau}^{\pi}\big\rangle\!\big\rangle\big|\leq C_2\lambda_{\sigma}^{p_2}\lambda_{\tau}^{p_2},
  \quad  \forall\, (\sigma,\tau)\in \Gamma\times \Gamma.
\end{equation*}
Take $p\geq \max\{p_1,p_2\}$. Then, using the above inequalities, we obtain the bound
\begin{equation*}
  \big|\big\langle\!\big\langle \alpha \Phi + \beta \Psi, \phi_{\sigma,\tau}^{\pi}\big\rangle\!\big\rangle\big|
  \leq (|\alpha|C_1 + |\beta|C_2)\lambda_{\sigma}^p\lambda_{\tau}^p,\quad  \forall\, (\sigma,\tau)\in \Gamma\times \Gamma,
\end{equation*}
which means that $\alpha \Phi + \beta \Psi$ is integrable with respect to $\pi$.
For all $(\sigma,\tau)\in \Gamma\times \Gamma$, a straightforward calculation yields
\begin{equation*}
\begin{split}
  \widehat{\int_{\Sigma}\big(\alpha \Phi + \beta \Psi\big)d\pi}(\sigma,\tau)
  & = \big\langle\!\big\langle \alpha \Phi + \beta \Psi, \phi_{\sigma,\tau}^{\pi}\big\rangle\!\big\rangle\\
  & = \alpha\big\langle\!\big\langle  \Phi, \phi_{\sigma,\tau}^{\pi}\big\rangle\!\big\rangle
      + \beta \big\langle\!\big\langle\Psi, \phi_{\sigma,\tau}^{\pi}\big\rangle\!\big\rangle\\
  & = \alpha \widehat{\int_{\Sigma}\Phi d\pi}(\sigma,\tau) + \beta \widehat{\int_{\Sigma}\Psi d\pi}(\sigma,\tau)\\
  & = \widehat{\Big[\alpha \int_{\Sigma}\Phi d\pi + \beta \int_{\Sigma}\Psi d\pi\Big]}(\sigma,\tau),
\end{split}
\end{equation*}
which implies that $\int_{\Sigma}\big(\alpha \Phi + \beta \Psi\big)d\pi =  \alpha\int_{\Sigma} \Phi d\pi  +  \beta\int_{\Sigma}\Psi d\pi$,
see the comments after Definition~\ref{def-3-1}.
\end{proof}

For $\xi\in \mathcal{S}$, we use $\xi\geq 0$ to mean that $\xi(\omega)\geq 0$ for $\mu$-a.a. $\omega\in \Sigma$.
For $\Phi\in \mathcal{S}^*$, we use $\Phi\geq 0$ to mean that $\langle\!\langle \Phi, \xi\rangle\!\rangle\geq 0$ for all $\xi \in \mathcal{S}$ with $\xi\geq 0$.
In that case, we also say that $\Phi$ is a positive Bernoulli generalized functional.

\begin{theorem}\label{thr-4-4}
Let $\Phi \in \mathcal{S}^*$ be integrable with respect to $\pi$. Suppose that $\Phi\geq 0$. Then, for all $\xi \in \mathcal{S}$,
it holds true that
\begin{equation}\label{eq-4-8}
  \big\langle\!\big\langle \big(\!\int_{\Sigma} \Phi d\pi\big)\overline{\xi}, \xi\big\rangle\!\big\rangle \geq 0.
\end{equation}
\end{theorem}

\begin{proof}
Let $\xi \in \mathcal{S}$ be given. For each positive integer $n\geq 1$, we set $\xi_n = \sum_{\sigma\in \Gamma\!_n}\langle Z_{\sigma}, \xi\rangle Z_{\sigma}$.
Then it is easy to see that
\begin{equation*}
\overline{\xi_n} = \sum_{\sigma\in \Gamma\!_n}\langle Z_{\sigma}, \overline{\xi}\rangle Z_{\sigma}.
\end{equation*}
By applying Lemma~\ref{lem-2-5}, we know that $\xi_n$ converges to $\xi$ in the topology of $\mathcal{S}$ as $n\rightarrow \infty$.
Similarly, $\overline{\xi_n}$ converges to $\overline{\xi}$ in the topology of $\mathcal{S}$. Thus
\begin{equation*}
  \big\langle\!\big\langle \big(\!\int_{\Sigma} \Phi d\pi\big)\overline{\xi}, \xi\big\rangle\!\big\rangle
  = \lim_{n\to \infty} \big\langle\!\big\langle \big(\!\int_{\Sigma} \Phi d\pi\big)\overline{\xi_n}, \xi_n\big\rangle\!\big\rangle.
\end{equation*}
For each $n\geq 1$, $\sum_{\sigma,\tau\in \Gamma\!_n}\overline{\langle Z_{\sigma}, \xi\rangle}\langle Z_{\tau}, \xi\rangle
\phi_{\sigma,\tau}^{\pi}$ obviously belongs to $\mathcal{S}$, and moreover, by Proposition~\ref{prop-4-2}, we further know that
\begin{equation*}
  \sum_{\sigma,\tau\in \Gamma\!_n}\overline{\langle Z_{\sigma}, \xi\rangle}\langle Z_{\tau}, \xi\rangle
      \phi_{\sigma,\tau}^{\pi} \geq 0,
\end{equation*}
which, together with the assumption $\Phi\geq 0$, gives
\begin{equation*}
  \big\langle\!\big\langle\Phi,\sum_{\sigma,\tau\in \Gamma\!_n}\overline{\langle Z_{\sigma}, \xi\rangle}\langle Z_{\tau}, \xi\rangle
      \phi_{\sigma,\tau}^{\pi}\big\rangle\!\big\rangle
  \geq 0.
\end{equation*}
On the other hand, for each $n\geq 1$, by a careful examination we find
\begin{equation*}
\begin{split}
\big\langle\!\big\langle \big(\!\int_{\Sigma} \Phi d\pi\big)\overline{\xi_n}, \xi_n\big\rangle\!\big\rangle
  &= \sum_{\sigma,\tau\in \Gamma\!_n}\overline{\langle Z_{\sigma}, \xi\rangle}\langle Z_{\tau}, \xi\rangle
   \widehat{\int_{\Sigma}\Phi d\pi}(\sigma, \tau)\\
  &= \sum_{\sigma,\tau\in \Gamma\!_n}\overline{\langle Z_{\sigma}, \xi\rangle}\langle Z_{\tau}, \xi\rangle
     \big\langle\!\big\langle\Phi, \phi_{\sigma,\tau}^{\pi}\big\rangle\!\big\rangle\\
  & = \big\langle\!\big\langle\Phi,\sum_{\sigma,\tau\in \Gamma\!_n}\overline{\langle Z_{\sigma}, \xi\rangle}\langle Z_{\tau}, \xi\rangle
      \phi_{\sigma,\tau}^{\pi}\big\rangle\!\big\rangle.
\end{split}
\end{equation*}
Thus $\big\langle\!\big\langle \big(\!\int_{\Sigma} \Phi d\pi\big)\overline{\xi_n}, \xi_n\big\rangle\!\big\rangle\geq 0$
for all $n\geq 0$, which directly leads to the desired result as follows
\begin{equation*}
 \big\langle\!\big\langle \big(\!\int_{\Sigma} \Phi d\pi\big)\overline{\xi}, \xi\big\rangle\!\big\rangle
  = \lim_{n\to \infty} \big\langle\!\big\langle \big(\!\int_{\Sigma} \Phi d\pi\big)\overline{\xi_n}, \xi_n\big\rangle\!\big\rangle
  \geq 0.
\end{equation*}
This completes the proof.
\end{proof}

A family $\{\Phi_{\alpha} \mid \alpha\in \Lambda\}$ of Bernoulli generalized functionals is said to be uniformly integrable
with respect to $\pi$ if there exist constants $C\geq 0$ and $p\geq 0$ such that
\begin{equation*}
 \sup_{\alpha \in \Lambda} \big|\big\langle\!\big\langle \Phi_{\alpha}, \phi_{\sigma,\tau}^{\pi}\big\rangle\!\big\rangle\big|\leq C\lambda_{\sigma}^p\lambda_{\tau}^p,\quad
  \forall\, (\sigma,\tau)\in \Gamma\times \Gamma.
\end{equation*}
The next result establishes a convergence theorem for spectral integrals of Bernoulli generalized functionals.

\begin{theorem}\label{thr-4-5}
Let $(\Phi_n)_{n\geq 1}\subset \mathcal{S}^*$ be a sequence of Bernoulli generalized functionals. Suppose that the following conditions are satisfied:
\begin{enumerate}
  \item[(1)] $\Phi_n$ converges weakly to $\Phi_0\in \mathcal{S}^*$ as $n\rightarrow \infty$, namely
  $\lim_{n\to \infty}\langle\!\langle\Phi_n, \xi\rangle\!\rangle = \langle\!\langle\Phi_0, \xi\rangle\!\rangle $ for all $\xi\in \mathcal{S}$;
  \item[(2)] $(\Phi_n)_{n\geq 1}$ is uniformly integrable with respect to $\pi$.
\end{enumerate}
Then $\Phi_0$ is also integrable with respect to $\pi$. Moreover, for all $\xi\in \mathcal{S}$,
$(\int_{\Sigma}\Phi_nd\pi)\xi$ converges strongly to $(\int_{\Sigma}\Phi_0d\pi)\xi$ as $n\rightarrow \infty$.
\end{theorem}

\begin{proof}
By the uniform integrability of $(\Phi_n)_{n\geq 1}$, there exist constants $C\geq 0$ and $p\geq 0$ such that
\begin{equation}\label{eq-4-9}
\sup_{n\geq 1} \big|\big\langle\!\big\langle \Phi_n, \phi_{\sigma,\tau}^{\pi}\big\rangle\!\big\rangle\big|\leq C\lambda_{\sigma}^p\lambda_{\tau}^p,\quad
  \forall\, (\sigma,\tau)\in \Gamma\times \Gamma.
\end{equation}
On the other hand, since $\Phi_n$ converges weakly to $\Phi_0$ as $n\rightarrow \infty$,
we have
\begin{equation}\label{eq-4-10}
\lim_{n\to \infty}\big\langle\!\big\langle \Phi_n, \phi_{\sigma,\tau}^{\pi}\big\rangle\!\big\rangle
= \big\langle\!\big\langle \Phi_0, \phi_{\sigma,\tau}^{\pi}\big\rangle\!\big\rangle
\end{equation}
for all $(\sigma,\tau)\in \Gamma\times \Gamma$. Thus
$\big|\big\langle\!\big\langle \Phi_0, \phi_{\sigma,\tau}^{\pi}\big\rangle\!\big\rangle\big|\leq C\lambda_{\sigma}^p\lambda_{\tau}^p$,
$\forall\, (\sigma,\tau)\in \Gamma\times \Gamma$, which implies that $\Phi_0$ is integrable with respect to $\pi$.
Now consider the sequence $\int_{\Sigma}\Phi_n d\pi$, $n\geq 1$. Clearly, we have
\begin{equation*}
  \sup_{n\geq 1}\Big|\widehat{\int_{\Sigma}\Phi_n d\pi}(\sigma,\tau)\Big|
   = \sup_{n\geq 1}\big|\big\langle\!\big\langle \Phi_n, \phi_{\sigma,\tau}^{\pi}\big\rangle\!\big\rangle\big|
  \leq  C\lambda_{\sigma}^p\lambda_{\tau}^p,\quad \forall\, (\sigma,\tau)\in \Gamma\times \Gamma.
\end{equation*}
Take $q>p+\frac{1}{2}$. Then, by Theorem~\ref{thr-3-2}, there exists a sequence
$\mathsf{T}_n\in \mathfrak{L}(\mathcal{S}_q,\mathcal{S}_q^*)$, $n\geq 1$, such that
\begin{equation}\label{eq-4-11}
  \mathsf{T}_n\xi = \Big(\int_{\Sigma}\Phi_n d\pi\Big)\xi,\quad \forall\,\xi\in \mathcal{S},\, n\geq 1
\end{equation}
and
\begin{equation}\label{eq-4-12}
\sup_{n\geq 1}\|\mathsf{T}_n\|_{\mathfrak{L}(\mathcal{S}_q,\mathcal{S}_q^*)}\leq C\sum_{\sigma\in \Gamma}\lambda_{\sigma}^{-2(q-p)}.
\end{equation}

Next we show that $\mathsf{T}_n\xi\rightarrow \mathsf{T}_0\xi$ in the norm $\|\cdot\|_{-q}$ of $\mathcal{S}_q^*$ for each $\xi \in \mathcal{S}_q$.
However, in view of (\ref{eq-4-12}) and the fact that $\{Z_{\sigma} \mid \sigma\in \Gamma\}$ is total in $\mathcal{S}_q$,
it suffices to prove that $\mathsf{T}_nZ_{\sigma}\rightarrow \mathsf{T}_0Z_{\sigma}$ in the norm $\|\cdot\|_{-q}$ of $\mathcal{S}_q^*$ for each $\sigma \in \Gamma$.
Let $\sigma \in \Gamma$ be given, then by Lemma~\ref{lem-2-6} we have
\begin{equation*}
  \|\mathsf{T}_nZ_{\sigma}-\mathsf{T}_0Z_{\sigma}\|_{-q}^2
  = \sum_{\tau\in \Gamma}\lambda_{\tau}^{-2q}|\langle\!\langle \mathsf{T}_nZ_{\sigma}-\mathsf{T}_0Z_{\sigma}, Z_{\tau}\rangle\!\rangle|^2,\quad n\geq 1.
\end{equation*}
For each $\tau\in \Gamma$, it follows from (\ref{eq-4-11}) and (\ref{eq-4-10}) that
\begin{equation*}
 \lim_{n\to \infty} \lambda_{\tau}^{-2q}|\langle\!\langle \mathsf{T}_nZ_{\sigma}-\mathsf{T}_0Z_{\sigma}, Z_{\tau}\rangle\!\rangle|^2
  = \lim_{n\to \infty}\lambda_{\tau}^{-2q}|\langle\!\langle \Phi_n, \phi_{\sigma, \tau}^{\pi}\rangle\!\rangle
    - \langle\!\langle \Phi_0, \phi_{\sigma, \tau}^{\pi}\rangle\!\rangle|^2
  = 0.
\end{equation*}
On the other hand, we note that $\sum_{\tau\in \Gamma}4C^2\lambda_{\sigma}^{2p}\lambda_{\tau}^{-2(q-p)}
 = 4C^2\lambda_{\sigma}^{2p}\sum_{\tau\in \Gamma}\lambda_{\tau}^{-2(q-p)} <\infty$, and by (\ref{eq-4-9}) we have
\begin{equation*}
  \sup_{n\geq 1}\lambda_{\tau}^{-2q}|\langle\!\langle \mathsf{T}_nZ_{\sigma}-\mathsf{T}_0Z_{\sigma}, Z_{\tau}\rangle\!\rangle|^2
  = \sup_{n\geq 1}\lambda_{\tau}^{-2q}|\langle\!\langle \Phi_n, \phi_{\sigma, \tau}^{\pi}\rangle\!\rangle
    - \langle\!\langle \Phi_0, \phi_{\sigma, \tau}^{\pi}\rangle\!\rangle|^2
  \leq 4C^2\lambda_{\sigma}^{2p}\lambda_{\tau}^{-2(q-p)},\ \ \tau\in \Gamma.
\end{equation*}
Thus, by the dominated convergence theorem, we come to
 \begin{equation*}
  \lim_{n\to \infty} \|\mathsf{T}_nZ_{\sigma}-\mathsf{T}_0Z_{\sigma}\|_{-q}^2
  = \lim_{n\to \infty}\sum_{\tau\in \Gamma}\lambda_{\tau}^{-2q}|\langle\!\langle \mathsf{T}_nZ_{\sigma}-\mathsf{T}_0Z_{\sigma}, Z_{\tau}\rangle\!\rangle|^2
  =0,
\end{equation*}
which implies that $\mathsf{T}_nZ_{\sigma}\rightarrow \mathsf{T}_0Z_{\sigma}$ in the norm $\|\cdot\|_{-q}$ of $\mathcal{S}_q^*$.

Finally, for any $\xi \in \mathcal{S}$, in view of (\ref{eq-4-11}), we have
\begin{equation*}
  \lim_{n\to \infty}\Big(\int_{\Sigma}\Phi_n d\pi\Big)\xi = \lim_{n\to \infty}\mathsf{T}_n\xi = \mathsf{T}_0\xi = \Big(\int_{\Sigma}\Phi_0 d\pi\Big)\xi
\end{equation*}
in the norm $\|\cdot\|_{-q}$, which implies that $\Big(\int_{\Sigma}\Phi_n d\pi\Big)\xi$ converges to $\Big(\int_{\Sigma}\Phi_0 d\pi\Big)\xi$
in the strong topology of $\mathcal{S}^*$ as $n\rightarrow \infty$.
\end{proof}

\section{Example and further results}\label{sec-5}

In the final section, we show an example of an $\mathcal{S}$-smooth spectral measure and Bernoulli generalized functionals that are integrable
with respect to this spectral measure. Some further results are also obtained.

Throughout this section, we further assume that the Bernoulli space $(\Sigma, \mathscr{A},\mu)$ is symmetric,
namely the sequence $(\theta_n)_{n\geq 0}$ that defines the measure $\mu$ (see (\ref{eq-2-2}) in Section~\ref{sec-2})
satisfies the following requirements
\begin{equation*}
  \theta_n=\frac{1}{2},\quad \forall\, n\geq 0.
\end{equation*}
In this case, one has $Z_n =\zeta_n$, $n\geq 0$, which implies that $Z_{\sigma}^2=1$ for all $\sigma\in \Gamma$.
For details about $Z_n$ and $\zeta_n$, see (\ref{eq-2-3}) and (\ref{eq-2-1}) in Section~\ref{sec-2}.

As in previous sections, $\mathfrak{L}(\mathcal{S},\mathcal{S}^*)$ denotes the space of all continuous linear operators from $\mathcal{S}$ to $\mathcal{S}^*$,
and, for $p\geq 0$, $\mathfrak{L}(\mathcal{S}_p,\mathcal{S}_p^*)$ denotes the Banach space of all bounded linear operators
from $\mathcal{S}_p$ to $\mathcal{S}_p^*$. Note that a linear operator $\mathsf{T}\colon \mathcal{S}_p\rightarrow \mathcal{S}_p^*$
is bounded if and only if it is continuous.

For each $E\in \mathscr{A}$, by putting $\pi_0(E)\xi = \mathbf{1}_{E}\xi$, $\xi\in \mathcal{L}^2$,
we get a projection operator $\pi_0(E)$ on $\mathcal{L}^2$, where $\mathbf{1}_{E}$ denotes the indicator of $E$
and $\mathbf{1}_{E}\xi$ means the usual product of functions $\mathbf{1}_{E}$ and $\xi$ on $\Sigma$.
It can be shown that the mapping $E\mapsto \pi_0(E)$ defines a spectral measure $\pi_0$
on the Bernoulli space $(\Sigma, \mathscr{A},\mu)$,
which we call \textbf{the canonical spectral measure} on $(\Sigma, \mathscr{A},\mu)$.

\begin{theorem}\label{thr-5-1}
The canonical spectral measure $\pi_0$ is $\mathcal{S}$-smooth and its numerical spectral density $\phi_{\sigma,\tau}^{\pi_0}$
associated with $(\sigma,\tau)\in \Gamma\times \Gamma$ takes the following form
\begin{equation}\label{eq-5-1}
  \phi_{\sigma,\tau}^{\pi_0}= Z_{\sigma\bigtriangleup \tau},
\end{equation}
where $\sigma\bigtriangleup \tau = (\sigma\setminus \tau)\cup (\tau\setminus \sigma)$ and
$Z_{\sigma\bigtriangleup \tau}$ is the corresponding basis vector of the canonical orthonormal basis for $\mathcal{L}^2$ (see (\ref{eq-2-6}) for details).
\end{theorem}

\begin{proof}
Take $(\sigma,\tau)\in \Gamma$. Then, $\sigma\bigtriangleup \tau\in \Gamma$, which together with Lemma~\ref{lem-2-5}
implies that $Z_{\sigma\bigtriangleup \tau}\in \mathcal{S}$.
On the other hand, by (\ref{eq-2-6}) and the property that $Z_{\gamma}^2=1$ for $\gamma\in \Gamma$, we have
\begin{equation*}
Z_{\sigma}Z_{\tau}=\Big(\prod_{k\in \sigma\setminus \tau}Z_k\Big)\Big(\prod_{k\in \sigma\cap \tau}Z_k\Big)^2\Big(\prod_{k\in \tau\setminus \sigma}Z_k\Big)
= Z_{\sigma\bigtriangleup \tau}Z_{\sigma\cap \tau}^2
= Z_{\sigma\bigtriangleup \tau},
\end{equation*}
which together with the definition of $\pi_0$ gives
\begin{equation*}
 \langle \pi_0(E)Z_{\sigma}, Z_{\tau}\rangle
 = \int_{\Sigma} \mathbf{1}_{E}Z_{\sigma}Z_{\tau}d\mu
 = \int_E Z_{\sigma}Z_{\tau}d\mu
 = \int_E Z_{\sigma\bigtriangleup \tau}d\mu,\quad \forall\, E\in \mathscr{A}.
\end{equation*}
Therefore, $\pi_0$ is $\mathcal{S}$-smooth and its numerical spectral density $\phi_{\sigma,\tau}^{\pi_0}$
associated with $(\sigma,\tau)\in \Gamma\times \Gamma$ is exactly $Z_{\sigma\bigtriangleup \tau}$.
\end{proof}

\begin{theorem}\label{thr-5-2}
Every $\Phi\in \mathcal{S}^*$ is integrable with respect to the canonical spectral measure $\pi_0$.
\end{theorem}

\begin{proof}
Let $\Phi\in \mathcal{S}^*$ be given. Then, there is some $p\geq 0$ such that $\Phi\in \mathcal{S}_p^*$. For all $(\sigma,\tau)\in \Gamma\times \Gamma$,
we have
\begin{equation*}
  |\langle\!\langle \Phi, \phi_{\sigma,\tau}^{\pi_0}\rangle\!\rangle|
  = |\langle\!\langle \Phi, Z_{\sigma\bigtriangleup\tau}\rangle\!\rangle|
  \leq \|\Phi\|_{-p}\|Z_{\sigma\bigtriangleup\tau}\|_p,
\end{equation*}
which together with $\|Z_{\sigma\bigtriangleup\tau}\|_p= \lambda_{\sigma\bigtriangleup\tau}^p\leq \lambda_{\sigma}^p\lambda_{\tau}^p$ yields
\begin{equation*}
  |\langle\!\langle \Phi, \phi_{\sigma,\tau}^{\pi_0}\rangle\!\rangle|\leq \|\Phi\|_{-p}\lambda_{\sigma}^p\lambda_{\tau}^p.
\end{equation*}
Therefore, by definition, $\Phi$ is integrable with respect to $\pi_0$.
\end{proof}

\begin{remark}\label{rem-5-1}
Recall that $\mathcal{S}$ is dense in $\mathcal{S}_p$ for each $p\geq 0$. Thus,
if $\mathsf{T}\colon (\mathcal{S}, \|\cdot\|_p)\rightarrow \mathcal{S}_p^*$ is a bounded linear operator,
then there exists a unique $\widetilde{\mathsf{T}}\in \mathfrak{L}(\mathcal{S}_p,\mathcal{S}_p^*)$ such that
$\widetilde{\mathsf{T}}\xi = \mathsf{T}\xi$ for all $\xi\in \mathcal{S}$
and
\begin{equation*}
  \big\|\widetilde{\mathsf{T}}\big\|_{\mathfrak{L}(\mathcal{S}_p,\mathcal{S}_p^*)}
  = \sup\{\,\|\mathsf{T}\xi\|_{-p} \,\mid\, \xi\in S,\, \|\xi\|_p=1\,\}.
\end{equation*}
In that case, we identify $\mathsf{T}$ with $\widetilde{\mathsf{T}}$, namely $\mathsf{T}=\widetilde{\mathsf{T}}$.
\end{remark}

According to Theorem~\ref{thr-5-1}, the integration with the canonical spectral measure $\pi_0$ defines a linear
mapping $\Phi\mapsto \int_{\Sigma}\Phi d\pi_0$ from $\mathcal{S}^*$ to $\mathfrak{L}(\mathcal{S},\mathcal{S}^*)$.
The next theorem shows that this mapping is continuous.

\begin{theorem}\label{thr-5-3}
Let $p\geq 0$ and $\Phi\in \mathcal{S}^*_p$ be given. Then, for $q>p+\frac{1}{2}$, $\int_{\Sigma}\Phi d\pi_0 \in \mathfrak{L}(\mathcal{S}_q,\mathcal{S}_q^*)$ and
moreover
\begin{equation}\label{eq-5-2}
  \Big\|\int_{\Sigma}\Phi d\pi_0\Big\|_{\mathfrak{L}(\mathcal{S}_q,\mathcal{S}_q^*)}
  \leq \Big[\sum_{\sigma\in \Gamma}\lambda_{\sigma}^{-2(q-p)}\Big]\|\Phi\|_{-p}.
\end{equation}
\end{theorem}

\begin{proof}
Write $\mathsf{T}=\int_{\Sigma}\Phi d\pi_0$. Then, from the proof of Theorem~\ref{thr-5-2}, we find that
\begin{equation*}
  |\widehat{\mathsf{T}}(\sigma,\tau)|
   = |\langle\!\langle \Phi, \phi_{\sigma,\tau}^{\pi_0}\rangle\!\rangle|
 \leq \|\Phi\|_{-p}\lambda_{\sigma}^p\lambda_{\tau}^p, \quad \forall\, (\sigma,\tau)\in \Gamma\times \Gamma.
\end{equation*}
Consequently, by Theorem~\ref{thr-3-2} and Remark~\ref{rem-5-1}, we know that $\mathsf{T}\in \mathfrak{L}(\mathcal{S}_q,\mathcal{S}_q^*)$ and
\begin{equation*}
  \|\mathsf{T}\|_{\mathfrak{L}(\mathcal{S}_q,\mathcal{S}_q^*)}
  \leq \|\Phi\|_{-p}\Big[\sum_{\sigma\in \Gamma}\lambda_{\sigma}^{-2(q-p)}\Big]
  =\Big[\sum_{\sigma\in \Gamma}\lambda_{\sigma}^{-2(q-p)}\Big]\|\Phi\|_{-p}.
\end{equation*}
This completes the proof.
\end{proof}

For Bernoulli generalized functionals $\Phi$, $\Psi\in \mathcal{S}^*$, their convolution $\Phi\ast\Psi\in \mathcal{S}^*$ is defined by
\begin{equation}\label{eq-5-3}
  \widehat{\Phi\ast\Psi}(\sigma) = \widehat{\Phi}(\sigma) \widehat{\Psi}(\sigma),\quad \sigma\in \Gamma,
\end{equation}
where $\widehat{\Phi}$ is the Fock transform of $\Phi$, which is defined by $\widehat{\Phi}(\sigma)= \langle\!\langle \Phi, Z_{\sigma}\rangle\!\rangle$,
$\sigma\in \Gamma$.
Similarly, for operators $\mathsf{T}_1$, $\mathsf{T}_2\in \mathfrak{L}(\mathcal{S},\mathcal{S}^*)$,
their convolution $\mathsf{T}_1\ast \mathsf{T}_2\in \mathfrak{L}(\mathcal{S},\mathcal{S}^*)$ is determined by
\begin{equation}\label{eq-5-4}
  \widehat{\mathsf{T}_1\ast \mathsf{T}_2}(\sigma,\tau) = \widehat{\mathsf{T}_1}(\sigma,\tau) \widehat{\mathsf{T}_2}(\sigma,\tau),\quad
  (\sigma,\tau)\in \Gamma\times \Gamma.
\end{equation}
See \cite{wang-chen-1} and \cite{wang-chen-2} for details about convolutions of generalized functionals of discrete-time normal martingale,
which include Bernoulli generalized functionals as a special case,
and about convolutions of operators on these functionals, respectively.

\begin{theorem}
For all $\Phi$, $\Psi\in \mathcal{S}^*$, it holds true that
\begin{equation}\label{eq-5-5}
  \int_{\Sigma}\Phi\ast \Psi d\pi_0 = \Big(\int_{\Sigma}\Phi d\pi_0\Big)\ast \Big(\int_{\Sigma}\Phi d\pi_0\Big).
\end{equation}
\end{theorem}

\begin{proof}
For all $(\sigma,\tau)\in \Gamma\times \Gamma$, in view of $\phi_{\sigma,\tau}^{\pi_0}=Z_{\sigma\bigtriangleup \tau}$, we have
\begin{equation*}
  \widehat{\int_{\Sigma}\Phi\ast \Psi d\pi_0}(\sigma,\tau)
   =  \langle\!\langle \Phi\ast \Psi, Z_{\sigma\bigtriangleup\tau}\rangle\!\rangle
  = \widehat{\Phi\ast \Psi}(\sigma\bigtriangleup\tau)
  = \widehat{\Phi}(\sigma\bigtriangleup\tau)\widehat{\Psi}(\sigma\bigtriangleup\tau)
  = \langle\!\langle \Phi,\sigma\bigtriangleup\tau\rangle\!\rangle \langle\!\langle \Psi,\sigma\bigtriangleup\tau\rangle\!\rangle
\end{equation*}
and
\begin{equation*}
  \widehat{\Big(\int_{\Sigma}\Phi d\pi_0\Big)\ast\Big(\int_{\Sigma}\Phi d\pi_0\Big)}(\sigma,\tau)
  = \widehat{\Big(\int_{\Sigma}\Phi d\pi_0\Big)}(\sigma,\tau) \widehat{\Big(\int_{\Sigma}\Phi d\pi_0\Big)}(\sigma,\tau)
  = \langle\!\langle \Phi,\sigma\bigtriangleup\tau\rangle\!\rangle \langle\!\langle \Psi,\sigma\bigtriangleup\tau\rangle\!\rangle,
\end{equation*}
which implies that $\int_{\Sigma}\Phi\ast \Psi d\pi_0=\big(\int_{\Sigma}\Phi d\pi_0\big)\ast\big(\int_{\Sigma}\Phi d\pi_0\big)$.
\end{proof}

For Bernoulli generalized functionals $\Phi$, $\Psi\in \mathcal{S}^*$, in the spirit of \cite{wang-zhang}, one can define the Wick product
$\Phi\diamond\Psi$, which belongs to $\mathcal{S}^*$ and satisfies
\begin{equation}\label{eq-5-6}
  \widehat{\Phi\diamond\Psi}(\sigma) = \sum_{\tau\subset \sigma}\widehat{\Phi}(\tau)\widehat{\Psi}(\sigma\setminus\tau),\quad \sigma \in \Gamma,
\end{equation}
where $\widehat{\Upsilon}$ denotes the Fock transform of a Bernoulli generalized functional $\Upsilon$,
and $\sum_{\tau\subset \sigma}$ means to sum for all subsets of $\sigma$. Comparing (\ref{eq-5-6}) and (\ref{eq-5-3}), one can see that
the Wick product $\Phi\diamond\Psi$ differs greatly from the convolution $\Phi\ast\Psi$.
The next proposition further shows that their spectral integrals can have quite different regularity.

\begin{proposition}
Let $\Phi$, $\Psi\in \mathcal{S}_p^*$ be Bernoulli generalized functionals with  $p\geq 0$. Then
\begin{equation}\label{eq-5-7}
\int_{\Sigma}\Phi\diamond\Psi d\pi_0\in \mathfrak{L}(\mathcal{S}_{p+2},\mathcal{S}_{p+2}^*),\quad
\mbox{while}\quad \int_{\Sigma}\Phi\ast\Psi d\pi_0 \in \mathfrak{L}(\mathcal{S}_{2p+1},\mathcal{S}_{2p+1}^*).
\end{equation}
\end{proposition}

\begin{proof}
According to Lemma~\ref{lem-2-6}, we have
\begin{equation*}
  \sum_{\tau\in \Gamma}\lambda_{\tau}^{-2p}\big|\widehat{\Phi}(\tau)\big|^2
  = \sum_{\tau\in \Gamma}\lambda_{\tau}^{-2p}|\langle\!\langle \Phi, Z_{\tau}\rangle\!\rangle|^2
  = \|\Phi\|_{-p}^2
  < \infty.
\end{equation*}
Similarly, we also have $\sum_{\tau\in \Gamma}\lambda_{\tau}^{-2p}\big|\widehat{\Psi}(\tau)\big|^2= \|\Psi\|_{-p}^2<\infty$.
Using these relations, we find
\begin{equation*}
\begin{split}
  \sum_{\sigma\in \Gamma}\lambda_{\sigma}^{-2(p+1)}|\langle\!\langle \Phi\diamond\Psi, Z_{\sigma}\rangle\!\rangle|^2
   & = \sum_{\sigma\in \Gamma}\lambda_{\sigma}^{-2(p+1)}\big|\widehat{\Phi\diamond\Psi}(\sigma)\big|^2\\
   & = \sum_{\sigma\in \Gamma}\lambda_{\sigma}^{-2}
       \Big|\sum_{\tau\subset \sigma}\lambda_{\tau}^{-p}\widehat{\Phi}(\tau)\lambda_{\sigma\setminus\tau}^{-p}\widehat{\Psi}(\sigma\setminus\tau)\Big|^2\\
   &\leq \sum_{\sigma\in \Gamma}\lambda_{\sigma}^{-2}
        \Big[\sum_{\tau\subset \sigma}\lambda_{\tau}^{-2p}|\widehat{\Phi}(\tau)|^2\Big]
        \Big[\sum_{\tau\subset \sigma}\lambda_{\sigma\setminus\tau}^{-2p}|\widehat{\Psi}(\sigma\setminus\tau)|^2\Big]\\
   &\leq \|\Phi\|_{-p}^2 \|\Psi\|_{-p}^2\sum_{\sigma\in \Gamma}\lambda_{\sigma}^{-2}\\
   & < \infty,
\end{split}
\end{equation*}
which, together with Lemma~\ref{lem-2-6}, implies that $\Phi\diamond\Psi\in S_{p+1}^*$. Thus, by using Theorem~\ref{thr-5-3}, we come to
the relation $\int_{\Sigma}\Phi\diamond\Psi d\pi_0\in \mathfrak{L}(\mathcal{S}_{p+2},\mathcal{S}_{p+2}^*)$.
Next, we prove the second relation of (\ref{eq-5-7}). In fact, we have
\begin{equation*}
 \begin{split}
  \sum_{\sigma\in \Gamma}\lambda_{\sigma}^{-4p}|\langle\!\langle \Phi\ast\Psi, Z_{\sigma}\rangle\!\rangle|^2
      & = \sum_{\sigma\in \Gamma}\lambda_{\sigma}^{-4p}\big|\widehat{\Phi\ast\Psi}(\sigma)\big|^2\\
      & = \sum_{\sigma\in \Gamma}\lambda_{\sigma}^{-2p}\big|\widehat{\Phi}(\sigma)\big|^2\lambda_{\sigma}^{-2p}\big|\widehat{\Psi}(\sigma)\big|^2\\
      & \leq \|\Phi\|_{-p}^2\|\Psi\|_{-p}^2,
\end{split}
\end{equation*}
which, together with Lemma~\ref{lem-2-6}, implies that $\Phi\ast\Psi\in S_{2p}^*$, which together with Theorem~\ref{thr-5-3} gives
the relation $\int_{\Sigma}\Phi\ast\Psi d\pi_0 \in \mathfrak{L}(\mathcal{S}_{2p+1},\mathcal{S}_{2p+1}^*)$.
\end{proof}

\section*{Acknowledgement}
The authors are extremely grateful to the referees for their valuable comments and suggestions on improvement
of the first version of the present paper.
This work is supported by National Natural Science Foundation of China (Grant No. 11861057).

\end{document}